\documentclass[12pt,reqno]{amsart}

\usepackage{amsfonts, amsthm, amsmath}
\allowdisplaybreaks[4]

\usepackage{rotating}

\usepackage{tikz}

\usepackage{graphics}

\usepackage{amssymb}

\usepackage{amscd}

\usepackage[latin2]{inputenc}

\usepackage{t1enc}

\usepackage[mathscr]{eucal}

\usepackage{indentfirst}

\usepackage{graphicx}

\usepackage{graphics}

\usepackage{pict2e}

\usepackage{mathrsfs}

\usepackage{enumerate}
\usepackage[pagebackref]{hyperref}
\hypersetup{backref, pagebackref, colorlinks=true}
\usepackage{cite}
\usepackage{color}
\usepackage{epic}
\usepackage{hyperref} 
\numberwithin{equation}{section}
\theoremstyle{plain}

\newtheorem{theorem}{Theorem}[section]

\newtheorem{lemma}[theorem]{Lemma}

\theoremstyle{definition}

\newtheorem{remark}[theorem]{Remark}

\newtheorem{?}[theorem]{Problem}

\def\boxit#1{\leavevmode\hbox{\vrule\vtop{\vbox{\kern.33333pt\hrule
    \kern1pt\hbox{\kern1pt\vbox{#1}\kern1pt}}\kern1pt\hrule}\vrule}}

\usepackage{collectbox}



\newcommand{\f}[1]{\ifthenelse{\equal{#1}{1}}{(q;q)_\infty}{(q^{#1};q^{#1})_{\infty}}}

\begin{document}
\title[Elementary proof of congruences modulo 25]{Elementary proof of congruences modulo 25 for broken $k$-diamond partitions}

\author[S. Chern]{Shane Chern}
\address[Shane Chern]{Department of Mathematics, The Pennsylvania State University, University Park, PA 16802, USA}
\email{shanechern@psu.edu}

\author[D. Tang]{Dazhao Tang}

\address[Dazhao Tang]{College of Mathematics and Statistics, Chongqing University, Huxi Campus LD206, Chongqing 401331, P.R. China}
\email{dazhaotang@sina.com}

\date{\today}

\begin{abstract}
Let $\Delta_{k}(n)$ denote the number of $k$-broken diamond partitions of $n$. Quite recently, the second author proved an infinite family of congruences modulo 25 for $\Delta_{k}(n)$ with the help of modular forms. In this paper, we aim to provide an elementary proof of this result.
\end{abstract}



%

\maketitle

\section{Introduction}\label{sec1}
The notion of broken $k$-diamond partitions was introduced by Andrews and Paule \cite{AP2007} in 2007. They showed that the generating function of $\Delta_{k}(n)$, the number of broken $k$-diamond partitions of $n$, is given by
\begin{align*}
\sum_{n=0}^{\infty}\Delta_{k}(n)q^{n}=\dfrac{(q^{2};q^{2})_{\infty}(q^{2k+1};q^{2k+1})_{\infty}}{(q;q)_{\infty}^3(q^{2(2k+1)};q^{2(2k+1
)})_{\infty}}.
\end{align*}

Throughout this paper, we assume that $|q|<1$ and adopt the customary $q$-series notation:
\begin{align*}
(a;q)_{\infty} &:=\prod_{n=0}^{\infty}(1-aq^{n}).
\end{align*}

The following two congruences modulo $5$ were obtained by S.H.~Chan \cite{Chan2008} and later rediscovered by Radu \cite{Rad2015}:
\begin{align*}
\Delta_{2}(25n+14) &\equiv0\pmod{5},\\
\Delta_{2}(25n+24) &\equiv0\pmod{5}.
\end{align*}
In fact, Chan extended these congruences to
\begin{align}
\Delta_{2}\left(5^{\alpha+1}n+\dfrac{11\times5^{\alpha}+1}{4}\right) &\equiv0\pmod{5},\label{mod 5 family-1}\\
\Delta_{2}\left(5^{\alpha+1}n+\dfrac{19\times5^{\alpha}+1}{4}\right) &\equiv0\pmod{5}.\label{mod 5 family-2}
\end{align}
Hirschhorn \cite{Hir2017} subsequently gave simple proofs of \eqref{mod 5 family-1} and \eqref{mod 5 family-2}.

Furthermore, other infinite families of congruences modulo $5$ satisfied by $\Delta_{2}(n)$ have been discovered by many authors. The interested readers may refer to Radu \cite{Rad2015} and Xia \cite{Xia2017}.

Quite recently, with the help of modular forms, the second author \cite[Theorem 2]{Tang2018} proved the following infinite family of congruences modulo 25 for $\Delta_{k}(n)$.

\begin{theorem}\label{THM:mod 25}
For all $n\geq0$,
\begin{align}
\Delta_{k}(125n+99) &\equiv0\pmod{25}, \quad \emph{if}\quad k\equiv62\pmod{125}.\label{con:mod 25}
\end{align}
\end{theorem}

Our main purpose of this paper is to provide an elementary proof of Theorem \ref{THM:mod 25}.

\section{Preliminaries}
For notational convenience, we denote
\begin{align*}
E_{j}=(q^{j};q^{j})_{\infty}.
\end{align*}
We also write
\begin{align*}
R(q) :=\dfrac{(q;q^{5})_{\infty}(q^{4};q^{5})_{\infty}}{(q^{2};q^{5})_{\infty}(q^{3};q^{5})_{\infty}}.
\end{align*}

From \cite[Eq.~(8.4.4)]{Hir}, one has the following $5$-dissection identity.

\begin{lemma}
\begin{align}
\dfrac{1}{E_{1}} &=\dfrac{E_{25}^{5}}{E_{5}^{6}}\Bigg(\dfrac{1}{R(q^{5})^{4}}+\dfrac{q}{R(q^{5})^{3}}+\dfrac{2q^{2}}{R(q^{5})^{2}}+\dfrac{3q^{3}}{R(q^{5})}+5q^{4}\notag\\
 &\quad\quad\quad\quad-3q^{5}R(q^{5})+2q^{6}R(q^{5})^{2}-q^{7}R(q^{5})^{3}+q^{8}R(q^{5})^{4}\Bigg).\label{Hir dissection}
\end{align}
\end{lemma}

We now absorb the ideas of \cite{BB2018} with some refinements. The first ingredient from \cite{BB2018} is the following three relations.

\begin{lemma}[Lemma 1.3, \cite{BB2018}]
Let $x=\dfrac{1}{R(q)}$ and $y=\dfrac{1}{R(q^{2})}$, then
\begin{align}
\dfrac{x^{2}}{y}-\dfrac{y}{x^{2}} &=\dfrac{4q}{K},\label{iden2}\\
xy^{2}-\dfrac{q^{2}}{xy^{2}} &=K,\label{iden1}\\
\dfrac{y^{3}}{x}+\dfrac{q^{2}x}{y^{3}} &=K-2q+\dfrac{4q^{2}}{K}.\label{iden3}
\end{align}
where $K=\dfrac{E_{2}E_{5}^{5}}{E_{1}E_{10}^{5}}$.
\end{lemma}\label{key lemma}

Now, for $\alpha\in\mathbb{Z}_{\ge 0}$ and $\beta\in\mathbb{Z}$, we define
\begin{equation}\label{eq:P-def}
P(\alpha,\beta):=x^{\alpha+2\beta}y^{2\alpha-\beta}+\dfrac{(-1)^{\alpha+\beta}q^{2\alpha}}{x^{\alpha+2\beta}y^{2\alpha-\beta}}.
\end{equation}
It is not hard to observe that
\begin{align}
P(0,0)&=2,\label{eq:p00}\\
P(0,1)&=\dfrac{x^{2}}{y}-\dfrac{y}{x^{2}} =\dfrac{4q}{K},\label{eq:p01}\\
P(1,0)&=xy^{2}-\dfrac{q^{2}}{xy^{2}} =K,\label{eq:p10}\\
P(1,-1)&=\dfrac{y^{3}}{x}+\dfrac{q^{2}x}{y^{3}} =K-2q+\dfrac{4q^{2}}{K}.\label{eq:p1-1}
\end{align}

With the help of the following recurrence relations along with the initial conditions \eqref{eq:p00}--\eqref{eq:p1-1}, one may easily express $P(\alpha,\beta)$ in terms of $K$ and $q$ for arbitrary $\alpha\in\mathbb{Z}_{\ge 0}$ and $\beta\in\mathbb{Z}$.

\begin{lemma}\label{eq:le:rec}
For $\alpha\in\mathbb{Z}_{\ge 0}$ and $\beta\in\mathbb{Z}$, 
\begin{align}
P(\alpha,\beta+1)&=\dfrac{4q}{K}P(\alpha,\beta)+P(\alpha,\beta-1),\label{eq:P-rec-1}\\
P(\alpha+2,0)&=KP(\alpha+1,0)+q^2 P(\alpha,0),\label{eq:P-rec-2}\\
P(\alpha+2,-1)&=\left(K-2q+\dfrac{4q^{2}}{K}\right)P(\alpha+1,0)-q^2 P(\alpha,1).\label{eq:P-rec-3}
\end{align}
\end{lemma}

\begin{proof}
We first notice that
\begin{align*}
P(\alpha,\beta)P(0,1)&=\left(x^{\alpha+2\beta}y^{2\alpha-\beta}+\dfrac{(-1)^{\alpha+\beta}q^{2\alpha}}{x^{\alpha+2\beta}y^{2\alpha-\beta}}\right)\left(\dfrac{x^{2}}{y}-\dfrac{y}{x^{2}}\right)\\
&=\left(x^{\alpha+2(\beta+1)}y^{2\alpha-(\beta+1)}+\dfrac{(-1)^{\alpha+(\beta+1)}q^{2\alpha}}{x^{\alpha+2(\beta+1)}y^{2\alpha-(\beta+1)}}\right)\\
&\qquad-\left(x^{\alpha+2(\beta-1)}y^{2\alpha-(\beta-1)}+\dfrac{(-1)^{\alpha+(\beta-1)}q^{2\alpha}}{x^{\alpha+2(\beta-1)}y^{2\alpha-(\beta-1)}}\right)\\
&=P(\alpha,\beta+1)-P(\alpha,\beta-1).
\end{align*}
This gives \eqref{eq:P-rec-1}.

Next, it follows from \eqref{eq:P-def} and \eqref{eq:p10} that
\begin{align*}
P(\alpha+1,0)P(1,0)&=\left(x^{\alpha+1}y^{2(\alpha+1)}+\dfrac{(-1)^{\alpha+1}q^{2(\alpha+1)}}{x^{\alpha+1}y^{2(\alpha+1)}}\right)\left(xy^{2}-\dfrac{q^{2}}{xy^{2}}\right)\\
&=\left(x^{\alpha+2}y^{2(\alpha+2)}+\dfrac{(-1)^{\alpha+2}q^{2(\alpha+2)}}{x^{\alpha+2}y^{2(\alpha+2)}}\right)
-q^2 \left(x^{\alpha}y^{2\alpha}+\dfrac{(-1)^{\alpha}q^{2\alpha}}{x^{\alpha}y^{2\alpha}}\right)\\
&=P(\alpha+2,0)-q^2 P(\alpha,0),
\end{align*}
which is equivalent to \eqref{eq:P-rec-2}.

At last, we have
\begin{align*}
P(\alpha+1,0)P(1,-1)&=\left(x^{\alpha+1}y^{2(\alpha+1)}+\dfrac{(-1)^{\alpha+1}q^{2(\alpha+1)}}{x^{\alpha+1}y^{2(\alpha+1)}}\right)\left(\dfrac{y^{3}}{x}+\dfrac{q^{2}x}{y^{3}}\right)\\
&=\left(x^{\alpha}y^{2\alpha+5}+\dfrac{(-1)^{\alpha+1}q^{2\alpha+4}}{x^{\alpha}y^{2\alpha+5}}\right)+ q^2 \left(x^{\alpha+2}y^{2\alpha-1}+\dfrac{(-1)^{\alpha+1}q^{2\alpha}}{x^{\alpha+2}y^{2\alpha-1}}\right)\\
&=P(\alpha+2,-1)+q^2 P(\alpha,1).
\end{align*}
This yields \eqref{eq:P-rec-3}.
\end{proof}

The other ingredient we request from \cite{BB2018} states as follows:

\begin{lemma}
If $K=\dfrac{E_{2}E_{5}^{5}}{E_{1}E_{10}^{5}}$, then
\begin{align}\label{simp iden}
K-3q-\dfrac{4q^{2}}{K}=\dfrac{E_{1}^{2}E_{2}^{2}}{E_{5}^{2}E_{10}^{2}}.
\end{align}
\end{lemma}

\begin{proof}
Eq.~\eqref{simp iden} is an immediate consequence of \cite[Eqs.~(1.7) and (2.2)]{BB2018}.
\end{proof}

\section{Proof of Theorem \ref{THM:mod 25}}

Let
\begin{align*}
\sum_{n=0}^{\infty}c(n)q^{n}=\dfrac{E_{2}}{E_{1}^{3}}.
\end{align*}
We shall show
\begin{theorem}\label{th:c125n+99}
For any integer $n\geq0$,
\begin{align}
c(125n+99) &\equiv0\pmod{25}.\label{con:mod 25-2}
\end{align}
\end{theorem}

One readily sees that Theorem \ref{THM:mod 25} is a direct consequence of \eqref{con:mod 25-2} since if $k\equiv 62 \pmod{125}$, then $2k+1$ is a multiple of $125$.

\subsection{The first 5-dissection: Identities of Baruan and Begum}

According to \cite[Eqs.~(2.8) and (2.9)]{BB2018}, we have
\begin{align}\label{gf:5n+4}
\sum_{n=0}^{\infty}c(5n+4)q^{n} &=41\dfrac{E_{10}^{3}}{E_{1}^2E_{2}^2E_{5}}+860q\dfrac{E_{10}^{6}}{E_{1}^{5}E_{2}E_{5}^{2}}+6800q^{2}\dfrac{E_{10}^{9}}{E_{1}^{8}E_{5}^{3}}\nonumber\\
 &\quad+24000q^{3}\dfrac{E_{2}E_{10}^{12}}{E_{1}^{11}E_{5}^{4}}+32000q^{4}\dfrac{E_{2}^{2}E_{10}^{15}}{E_{1}^{14}E_{5}^{5}}.
\end{align}

Notice that the coefficients $6800$, $24000$ and $32000$ are multiples of $25$. Moreover,
$$860q\dfrac{E_{10}^{6}}{E_{1}^{5}E_{2}E_{5}^{2}}\equiv 10q\dfrac{E_{10}^6}{E_2 E_5^3} \pmod{25},$$
which contains no terms of the form $q^{5n+4}$.

Let
\begin{align}
\sum_{n=0}^{\infty}\tilde{c}(n)q^{n}=\dfrac{E_{10}^{3}}{E_{1}^2E_{2}^2E_{5}}.\label{eq:c1}
\end{align}
To show \eqref{con:mod 25-2}, it suffices to prove that
\begin{align}\label{cong:c1 mod 25}
\tilde{c}(25n+19)\equiv0\pmod{25}.
\end{align}

\subsection{The second 5-dissection: Cubic partition pairs}

We now observe that
\begin{equation}
\dfrac{1}{E_1^2E_2^2}=:\sum_{n=0}^\infty b(n)q^n\label{eq:cpp}
\end{equation}
can be treated as the generating function of \textit{cubic partition pairs} (cf.~\cite{ZZ2011}). In particular, Zhao and Zhong \cite{ZZ2011} proved that
\begin{equation}\label{eq:cpp-mod5}
b(5n+4)\equiv 0 \pmod{5}.
\end{equation}
For other interesting arithmetic properties of $b(n)$, we refer to \cite{Che2017,Hir2018,Ki2011,Li2017,LWX2018,Zh2012}.

The following dissection identity will play an important role.

\begin{lemma}\label{le:cpp-dis}
We have
\begin{align}
\sum_{n=0}^{\infty}b(5n+4)q^{n}&=35\dfrac{E_{5}^{2}E_{10}^{2}}{E_{1}^{4}E_{2}^{4}}+700q \dfrac{E_{5}^{4}E_{10}^{4}}{E_{1}^{6}E_{2}^{6}}+6875 q^2 \dfrac{E_{5}^{6}E_{10}^{6}}{E_{1}^{8}E_{2}^{8}}\notag\\
&\quad +31250 q^3 \dfrac{E_{5}^{8}E_{10}^{8}}{E_{1}^{10}E_{2}^{10}}+78125 q^4 \dfrac{E_{5}^{10}E_{10}^{10}}{E_{1}^{12}E_{2}^{12}}.\label{eq:le:cpp-dis}
\end{align}
\end{lemma}

\begin{proof}
We know from \eqref{Hir dissection} that
\begin{align*}
\sum_{n=0}^{\infty}b(n)q^{n} &=\dfrac{E_{25}^{10}E_{50}^{10}}{E_{5}^{12}E_{10}^{12}}\Bigg(\dfrac{1}{R(q^{5})^{4}}+\dfrac{q}{R(q^{5})^{3}}+\dfrac{2q^{2}}{R(q^{5})^{2}}+\dfrac{3q^{3}}{R(q^{5})}+5q^{4}\\
 &\quad\quad\quad\quad\quad\quad-3q^{5}R(q^{5})+2q^{6}R(q^{5})^{2}-q^{7}R(q^{5})^{3}+q^{8}R(q^{5})^{4}\Bigg)^{2}\\
 &\quad\times\Bigg(\dfrac{1}{R(q^{10})^{4}}+\dfrac{q^{2}}{R(q^{10})^{3}}+\dfrac{2q^{4}}{R(q^{10})^{2}}+\dfrac{3q^{6}}{R(q^{10})}+5q^{8}\\
 &\quad\quad\quad-3q^{10}R(q^{10})+2q^{12}R(q^{10})^{2}-q^{14}R(q^{10})^{3}+q^{16}R(q^{10})^{4}\Bigg)^{2}.
\end{align*}

Extracting terms of the form $q^{5n+4}$, dividing by $q^4$ and replacing $q^5$ by $q$ yields
\begin{align*}
\sum_{n=0}^{\infty} &b(5n+4)q^{n}=
\dfrac{E_{5}^{10}E_{10}^{10}}{E_{1}^{12}E_{2}^{12}}\Bigg(5\left(x^{8}y^{6}+\dfrac{q^{8}}{x^{8}y^{6}}\right)
+10\left(x^{6}y^{7}-\dfrac{q^{8}}{x^{6}y^{7}}\right)\\
 &\quad+20\left(x^{4}y^{8}+\dfrac{q^{8}}{x^{4}y^{8}}\right)+40q\left(x^{7}y^{4}-\dfrac{q^{6}}{x^{7}y^{4}}\right)+100q\left(x^{5}y^{5}-\dfrac{q^{6}}{x^{5}y^{5}}\right)\\
 &\quad+80q\left(x^{3}y^{6}-\dfrac{q^{6}}{x^{3}y^{6}}\right)+40q\left(xy^{7}+\dfrac{q^{6}}{xy^{7}}\right)+20q\left(-\dfrac{y^{8}}{x}+\dfrac{q^{6}x}{y^{8}}\right)\\
 &\quad+20q^{2}\left(x^{8}y-\dfrac{q^{4}}{x^{8}y}\right)+135q^{2}\left(x^{6}y^{2}+\dfrac{q^{4}}{x^{6}y}\right)+320q^{2}\left(x^{4}y^{3}-\dfrac{q^{4}}{x^{4}y^{3}}\right)\\
 &\quad+540q^{2}\left(x^{2}y^{4}+\dfrac{q^{4}}{x^{2}y^{4}}\right)+150q^{2}\left(y^{5}-\dfrac{q^{4}}{y^{5}}\right)
 +135q^{2}\left(\dfrac{y^{6}}{x^{2}}+\dfrac{q^{4}x^{2}}{y^{6}}\right)\\
 &\quad+40q^{2}\left(\dfrac{y^{7}}{x^{4}}-q^{4}\dfrac{x^{4}}{y^{7}}\right)+5q^{2}\left(\dfrac{y^{8}}{x^{6}}+\dfrac{q^{4}x^{6}}{y^{8}}\right)
 -40q^{3}\left(\dfrac{x^{7}}{y}+\dfrac{q^{2}y}{x^{7}}\right)\\
 &\quad+150q^{3}\left(x^{5}-\dfrac{q^{2}}{x^{5}}\right)+320q^{3}\left(x^{3}y+\dfrac{q^{2}}{x^{3}y}\right)+540q^{3}\left(xy^{2}-\dfrac{q^{2}}{xy^{2}}\right)\\
 &\quad-320q^{3}\left(\dfrac{y^{3}}{x}+\dfrac{q^{2}x}{y^{3}}\right)+320q^{3}\left(-\dfrac{y^{4}}{x^{3}}+\dfrac{q^{2}x^{3}}{y^{4}}\right)
 -100q^{3}\left(\dfrac{y^{5}}{x^{5}}+\dfrac{q^{2}x^{5}}{y^{5}}\right)\\
 &\quad+10q^{3}\left(-\dfrac{y^{6}}{x^{7}}+\dfrac{q^{2}x^{7}}{y^{6}}\right)+20q^{4}\left(\dfrac{y^{4}}{x^{8}}+\dfrac{x^{8}}{y^{4}}\right)
 +801q^{4}\left(-\dfrac{x^{6}}{y^{3}}+\dfrac{y^{3}}{x^{6}}\right)\\
 &\quad+540q^{4}\left(\dfrac{y^{2}}{x^{4}}+\dfrac{x^{4}}{y^{2}}\right)+540q^{4}\left(-\dfrac{x^{2}}{y}+\dfrac{y}{x^{2}}\right)+225q^{4}\Bigg),
\end{align*}
where $x$ and $y$ are as defined in Lemma \ref{key lemma}.

In view of \eqref{eq:P-def}, we may rewrite the above identity as
\begin{align*}
\dfrac{E_{1}^{12}E_{2}^{12}}{E_{5}^{10}E_{10}^{10}} &\sum_{n=0}^{\infty}b(5n+4)q^{n}
 =\Big(5P(4,2)+10P(4,1)+20P(4,0)\Big)\\
&\quad+q\Big(40P(3,2)+100P(3,1)+80P(3,0)+40P(3,-1)-20P(3,-2)\Big)\\
&\quad+q^2\Big(20P(2,3)+135P(2,2)+320P(2,1)+540P(2,0)+150P(2,-1)\\
&\quad\qquad\qquad+135P(2,-2)+40P(2,-3)+5P(2,-4)\Big)\\
&\quad+q^3\Big(-40P(1,3)+150P(1,2)+320P(1,1)+540P(1,0)-320P(1,-1)\\
&\quad-320P(1,-2)-100P(1,-3)-10P(1,-4)\Big)\\
&\quad+q^4\Big(20P(0,4)-80P(0,3)+540P(0,2)-540P(0,1)+225\Big).
\end{align*}

Using Lemma \ref{eq:le:rec} and the initial conditions \eqref{eq:p00}--\eqref{eq:p1-1} to express each summand $P(\cdot,\cdot)$ in terms of $K$ and $q$, we may further simplify the above identity as
\begin{align*}
\dfrac{E_{1}^{12}E_{2}^{12}}{E_{5}^{10}E_{10}^{10}} &\sum_{n=0}^{\infty}b(5n+4)q^{n}=35 K^4 + 280 K^3 q + 1905 K^2 q^2 + 1760 K q^3 + 13825 q^4\\
&\quad\quad - \frac{7040 q^5}{K} + \frac{30480 q^6}{K^2} - \frac{17920 q^7}{K^3} + \frac{8960 q^8}{K^4}\\
&\quad=35\left(K-3q-\dfrac{4q^{2}}{K}\right)^4+700q \left(K-3q-\dfrac{4q^{2}}{K}\right)^3\\
&\quad\quad+6875 q^2 \left(K-3q-\dfrac{4q^{2}}{K}\right)^2 +31250 q^3 \left(K-3q-\dfrac{4q^{2}}{K}\right)+78125 q^4\\
&\quad=35\dfrac{E_{1}^{8}E_{2}^{8}}{E_{5}^{8}E_{10}^{8}}+700q \dfrac{E_{1}^{6}E_{2}^{6}}{E_{5}^{6}E_{10}^{6}}+6875 q^2 \dfrac{E_{1}^{4}E_{2}^{4}}{E_{5}^{4}E_{10}^{4}} +31250 q^3 \dfrac{E_{1}^{2}E_{2}^{2}}{E_{5}^{2}E_{10}^{2}}+78125 q^4,
\end{align*}
where we use \eqref{simp iden} in the last identity. Lemma \ref{le:cpp-dis} follows readily.
\end{proof}

\begin{remark}
It is easy to see that \eqref{eq:cpp-mod5} is a direct consequence of Lemma \ref{le:cpp-dis}.
\end{remark}

\begin{remark}
Let $a(n)$ count the number of \textit{cubic partitions} of $n$, which were introduced by H.C.~Chan \cite{Ch2010}. Its generating function is
\begin{align*}
\sum_{n=0}^{\infty}a(n)q^{n}=\dfrac{1}{E_{1}E_{2}}.
\end{align*}
Using modular forms, H.H.~Chan and Toh \cite{CT2010} and independently Xiong \cite{Xiong2011} found an infinite family of congruences modulo powers of $5$ for $a(n)$. We notice that, by similar arguments to the proof of Lemma \ref{le:cpp-dis}, it is not hard to prove that
\begin{align}\label{cubic gf:5n+2}
\sum_{n=0}^{\infty}a(5n+2)q^{n} &=3\dfrac{E_{5}E_{10}}{E_{1}^{2}E_{2}^{2}}+25q\dfrac{E_{5}^{3}E_{10}^{3}}{E_{1}^{4}E_{2}^{4}}
+125q^{2}\dfrac{E_{5}^{5}E_{10}^{5}}{E_{1}^{6}E_{2}^{6}}.
\end{align}
This gives an elementary proof of an identity due to Xiong \cite{Xiong2011}. We also learnt from Michael Hirschhorn that he was able to give an completely elementary proof \cite{Hir2018b} of the infinite family of congruences modulo powers of $5$ for $a(n)$ obtained by H.H.~Chan and Toh as well as Xiong.
\end{remark}

In view of \eqref{eq:c1} and \eqref{eq:le:cpp-dis}, we have
\begin{align}
\sum_{n=0}^{\infty}\tilde{c}(5n+4)q^{n}&=35\dfrac{E_{5}^{2}E_{10}^{2}}{E_{1}^{5}E_{2}}+700q \dfrac{E_{5}^{4}E_{10}^{4}}{E_{1}^{7}E_{2}^{3}}+6875 q^2 \dfrac{E_{5}^{6}E_{10}^{6}}{E_{1}^{9}E_{2}^{5}}\notag\\
&\quad +31250 q^3 \dfrac{E_{5}^{8}E_{10}^{8}}{E_{1}^{11}E_{2}^{7}}+78125 q^4 \dfrac{E_{5}^{10}E_{10}^{10}}{E_{1}^{13}E_{2}^{9}}.\label{eq:c-1-diss}
\end{align}

\subsection{The final punch}

Now \eqref{cong:c1 mod 25} is almost trivial from \eqref{eq:c-1-diss}. We have
\begin{align*}
\sum_{n=0}^{\infty}\tilde{c}(5n+4)q^{n}\equiv 10\dfrac{E_{5}^{2}E_{10}^{2}}{E_{1}^{5}E_{2}}&\equiv 10 \dfrac{E_{5}E_{10}^{2}}{E_{2}} \pmod{25},
\end{align*}
which contains no terms of the form $q^{5n+3}$. We therefore arrive at \eqref{cong:c1 mod 25}.

Consequently, we have
\begin{align*}
c(125n+99) \equiv0\pmod{25},
\end{align*}
and hence complete the proof of Theorem \ref{th:c125n+99}.

\section{Closing remarks}

In light of \eqref{gf:5n+4}, one has
\begin{align}\label{gf:5n+4-mod 5}
\sum_{n=0}^{\infty}c(5n+4)q^{n} &\equiv\dfrac{E_{10}^{3}}{E_{1}^2E_{2}^2E_{5}}\equiv\dfrac{E_{1}^{3}E_{2}^{3}E_{10}^{2}}{E_{5}}\pmod{5}.
\end{align}
It follows from \cite[Eq.~(9)]{Tang2018} that
\begin{align*}
c(25n+24)\equiv0\pmod{5},
\end{align*}
which is equivalent to
\begin{align*}
\Delta_{k}(25n+24) &\equiv0\pmod{5}, \quad \text{if}\quad k\equiv12\pmod{25}.
\end{align*}
This is discovered by the second author \cite[Theorem 1]{Tang2018}.

Moreover, we have
\begin{align*}
\sum_{n=0}^{\infty}\Delta_{2}(5n+4)q^{n} &=41\dfrac{E_{10}^{3}}{E_{1}E_{2}^{3}E_{5}}+860q\dfrac{E_{10}^{6}}{E_{1}^{4}E_{2}^{2}E_{5}^{2}}+6800q^{2}\dfrac{E_{10}^{9}}{E_{1}^{7}E_{2}E_{5}^{3}}\nonumber\\
 &\quad+24000q^{3}\dfrac{E_{10}^{12}}{E_{1}^{12}E_{5}^{4}}+32000q^{4}\dfrac{E_{2}E_{10}^{15}}{E_{1}^{15}E_{5}^{5}},
\end{align*}
from which it follows that
\begin{align*}
\sum_{n=0}^{\infty}\Delta_{2}(5n+4)q^{n} &\equiv41\dfrac{E_{10}^{3}}{E_{1}E_{2}^{3}E_{5}}=41\dfrac{E_{2}^{2}}{E_{1}}\dfrac{E_{10}}{E_{2}^{5}}\dfrac{E_{10}^{2}}{E_{5}}\equiv \psi(q)\psi(q^{5})\pmod{5},
\end{align*}
where $\psi(q):=\sum\limits_{n=0}^{\infty}q^{n(n+1)/2}$ is one of Ramanujan's classical theta functions. This is first proved by Hirschhorn \cite[Eq.~(2.9)]{Hir2017}.

\section{Acknowledgements}
The second author was supported by the National Natural Science Foundation of China (No.~11501061).

\end{document}